\documentclass[12pt, a4paper]{article}

\usepackage{indentfirst}
\usepackage[leqno]{amsmath}
\usepackage{amssymb}
\usepackage{amsthm}
\usepackage{enumerate}
\usepackage{cases}
\usepackage[center]{titlesec}
\usepackage{fancyhdr}
\usepackage[colorlinks, 
linkcolor=blue,
anchorcolor=blue, 
citecolor=red]{hyperref}

\numberwithin{equation}{section} 
\newtheorem{lem}{Lemma}[section]
\newtheorem{thm}[lem]{Theorem}

\title{Hardy type inequalities on manifolds with nonnegative Ricci curvature}
\date{}

\begin{document}
	\author{Yuxin Dong, Hezi Lin, Lingen Lu}
	\maketitle

	\begin{abstract}
		ABSTRACT. We prove the Heisenberg-Pauli-Weyl inequality, Hardy-Sobolev inequality, and Caffarelli-Kohn-Nirenberg (CKN) inequality on manifolds with nonnegative Ricci curvature and Euclidean volume growth, of dimension $n\ge3$.
	\end{abstract}
	
	\section{Introduction}
	Let $n\ge 3$, the classical Hardy inequality in Euclidean space $\mathbb{R}^n$ states that 
	$$
	\Big(\frac{n-p}{p}\Big)^p\int_{\mathbb{R}^n}\frac{|u|^p}{|x|^p}~dx\le
	\int_{\mathbb{R}^n}|Du|^p~dx,
	$$ 
	where $1<p<n$ and $u\in C_c^\infty(\mathbb{R}^n)$. Moreover, the constant $\Big(\frac{n-p}{p}\Big)^p$ is sharp and never attained. In order to study the stability of some special semilinear elliptic equations, Brezis and V\'azquez \cite{BV} improve the Hardy inequality with a sharp remainder term for bounded domains. The remainder term problem of hardy inequality has attracted much attention, we refer the reader to \cite{GGM, GM1, GM2}. The hardy inequality on manifolds has also been well investigated, such as, \cite{KMM, LW, KO, BMV, FLL, CM}.

	Interpolating between the Hardy inequality and Sobolev inequality, one can obtain the following Hardy-Sobolev inequality 
	$$
	C_{HS}(s,p)(\int_{\mathbb{R}^n}\frac{|u|^{\frac{n-s}{n-p}p}}{|x|^s})^{\frac{n-p}{n-s}}
	\le\int_{\mathbb{R}^n}|Du|^p,\quad\forall u\in C_c^\infty(\mathbb{R}^n),
	$$ 
	where $1<p<n$, $0\le s\le p$ and $C_{HS}(s,p)>0$ is the sharp constant depending on $p$ and $s$. Note that, when $s=0$ (resp., $s=p$), the Hardy-Sobolev inequality is just the Sobolev (resp., Hardy) inequality. For the sharp constant of $C_{HS}(s,p)$, one can refer to \cite{L}.

	To generalize many important functional inequalities in Euclidean space, Caffarelli, Kohn and Nirenberg gave 
	$$
	C(a,b)(\int_{\mathbb{R}^n} |x|^{-bp}|u|^p)^\frac2p\le \int_{\mathbb{R}^n} |x|^{-2a}|Du|^2,
	$$ 
	where $-\infty<a<\frac{n-2}{2}, a\le b\le a+1$, $p=\frac{2n}{n-2+2(b-a)}$, and $C(a,b)>0$ is the sharp constant depending on $a$ and $b$. For more about the sharp constant $C(a,b)$, we refer to \cite{CC, CW}. When $0\le a<\frac{n-2}{2}, a\le b<a+1$, metric and topological rigidity theorems have been obtained in  \cite{dCC} for complete noncompact manifolds of nonnegative Ricci curvature supporting the above sharp CKN inequality.

	The Hardy inequality, Hardy-Sobolev inequality, and CKN inequality in Euclidean spaces constitute essential tools in the analysis and the study of PDEs. Here, we are going to generalize them to a manifold with nonnegative Ricci curvature. Inspired by the ABP method in Cabr\'e \cite{Ca}, Brendle \cite{Br} gave a sharp isoperimetric inequality in manifolds with nonnegative Ricci curvature, which is an important progress in this direction. Through the sharp isoperimetric inequality developed in \cite{Br}, Balogh and Krist\'aly \cite{BK} gave the P\'olya-Szeg\"o inequality in manifolds with nonnegative Ricci curvature. With the P\'olya-Szeg\"o inequality in \cite{BK} and some symmetrization techniques, we give the Heisenberg-Pauli-Weyl inequality, Hardy-Sobolev inequality, and CKN inequality in manifolds with nonnegative Ricci curvature.
	
	\section{Preliminaries}
	Let $(M^n,g), n\ge3$ be a complete noncompact Riemannian manifold with nonnegative Ricci curvature. Denote the asymptotic volume ratio of $M$ by
	$$
	\theta:=\lim_{r\to+\infty}\frac{|\{q\in M|d(o,q)<r\}|}{\omega_nr^n},
	$$
	where $o$ is an arbitrary fixed point in $M$ and $\omega_n$ is the unit ball volume in $\mathbb{R}^n$. The Bishop-Gromov volume comparison theorem implies that $\theta\le1$. If $\theta>0$, then we say that $M$ has Euclidean volume growth.
	
	Now, we follow the notation in \cite{BK}. Let $u:M\to\mathbb{R}$ be a fast decaying function, i.e. $|\{q\in M:|u(q)|>t\}|<\infty$ for every $t>0$. For an arbitrary fast decaying function, we associate its Euclidean rearrangement function $u^\star:\mathbb{R}^n\to[0,\infty)$, which is radially symmetric, nonincreasing in $|x|$, and defined by
	$$
|	\{x\in\mathbb{R}^n: u^\star(x)>t\}|=|\{q\in M:|u(q)|>t\}|.
	$$
	for every $t>0$. By the above equation and the layer cake representation, it is easy to show that
	$$
	\|u\|_{L^p(M)}=\|u^\star\|_{L^p(\mathbb{R}^n)},\quad p\in[1,\infty).
	$$
	Through the sharp isoperimetric inequality in \cite{Br}, Balogh and Krist\'aly gave the following P\'olya-Szeg\"o inequality.

	\begin{thm}{\cite{BK}}\label{PS}
		Let $(M^n,g)$ be a complete noncompact Riemannian manifold with nonnegative Ricci curvature and $u:M\to\mathbb{R}$ be a fast decaying function such that $|Du|\in L^p(M),p>1$. Then one has	
		$$
		\|Du\|_{L^p(M)}\ge\theta^{\frac1n}\|Du^\star\|_{L^p(\mathbb{R}^n)},
		$$
		where $\theta$ denote the asymptotic volume ratio of $M$.
	\end{thm}
	
	Now, we give two simple rearrangement inequalities for later use.
	\begin{lem}\label{sym1}
 		Let $P:(0,+\infty)\to(0,+\infty)$ be a nonincreasing continuous function and $u\in C_c^\infty(M)$. Then we have
 		$$
		\int_M |u|^pP(r)\le\int_{\mathbb{R}^n}(u^\star)^p P(|x|),
		\quad p\in[1,\infty),
		$$
		where $r(q)=d(o,q)$ is the distance function in $M$.
	\end{lem}
	
	\begin{proof}
		Through the layer cake representation, we have
		$$
		\begin{aligned}
		&\int_M |u|^pP(r)=\int_0^\infty\int_0^\infty\int_M\chi_{\{|u|^p> t\}}\chi_{\{P(r)> s\}}~dtds,\\
		&\int_{\mathbb{R}^n}(u^\star)^p P(|x|)=\int_0^\infty\int_0^\infty\int_{\mathbb{R}^n}\chi_{\{(u^\star)^p> t\}}\chi_{\{P(|x|)> s\}}~dtds,
		\end{aligned}
		$$
		where $\chi$ is the characteristic function of domain. The Bishop-Gromov volume comparison theorem and the decreasing of $P$ implies
		$$
		|\{q\in M|P(r(q))> s\}|\le|\{x\in\mathbb{R}^n|P(|x|)>s\}|,
		$$
		for fixed $s>0$. Note that $\{x\in\mathbb{R}^n|(u^\star)^p>t\}$ and $\{x\in\mathbb{R}^n|P(|x|)>s\}$ is two ball with same center, for fixed $t,s>0$. Then we have
		$$
		\int_M\chi_{\{|u|^p> t\}}\chi_{\{P(r)> s\}}\le
		\int_{\mathbb{R}^n}\chi_{\{(u^\star)^q> t\}}\chi_{\{P(|x|)> s\}}
		$$
		for $t,s>0$, which complete the proof.
	\end{proof}
	
	\begin{lem}\label{sym2}
		Let $P:(0,+\infty)\to(0,+\infty)$ be an increasing continuous function and $u\in C_c^\infty(M)$. Then we have
		$$
		\int_M |u|^pP(r)\ge\int_{\mathbb{R}^n}(u^\star)^p P(|x|),
		\quad p\in[1,\infty),
		$$		
		where $r(q)=d(o,q)$ is the distance function in $M$.
	\end{lem}

	\begin{proof}
		Through the layer cake representation, we have
		$$
		\begin{aligned}
			\begin{aligned}
				\int_M |u|^pP(r)&=\int_0^\infty\int_0^\infty\int_M\chi_{\{|u|^p> t\}}\chi_{\{P(r)> s\}}~dtds\\
				&=\int_0^\infty\int_0^\infty\int_M[\chi_{\{|u|^p> t\}}-\chi_{\{|u|^p> t\}}\chi_{\{P(r)\le s\}}]~dtds,\\
				\int_{\mathbb{R}^n}(u^\star)^p P(|x|)&=\int_0^\infty\int_0^\infty\int_{\mathbb{R}^n}\chi_{\{(u^\star)^p> t\}}\chi_{\{P(|x|)> s\}}~dtds\\
				&=\int_0^\infty\int_0^\infty\int_{\mathbb{R}^n}[\chi_{\{(u^\star)^p> t\}}-\chi_{\{(u^\star)^p> t\}}\chi_{\{P(|x|)\le s\}}]~dtds,
			\end{aligned}
		\end{aligned}
		$$
		where $\chi$ is the characteristic function of domain. The Bishop-Gromov volume comparison theorem and the increasing of $P$ implies
		$$
		|\{q\in M|P(r(q))\le s\}|\le|\{x\in\mathbb{R}^n|P(|x|)\le s\}|,
		$$
		for fixed $s>0$. Note that $\{x\in\mathbb{R}^n|(u^\star)^p>t\}$ and $\{x\in\mathbb{R}^n|P(|x|)\le s\}$ are two balls with the same center, for fixed $t,s>0$. Then we have
		$$
		\int_M\chi_{\{|u|^p> t\}}\chi_{\{P(r)\le s\}}\le
		\int_{\mathbb{R}^n}\chi_{\{(u^\star)^q> t\}}\chi_{\{P(|x|)\le s\}}
		$$
		for $t,s>0$, which completes the proof.
	\end{proof}
	
	\section{Hardy inequality and Heisenberg-Pauli-Weyl inequality}
	
	The following theorem was proved in \cite[Theorem 1.3]{KMM} on Riemannian-Finsler manifolds with nonnegative Ricci curvature, we are giving the same proof of this statement here for completeness.
	
	\begin{thm}\cite[Theorem 1.3]{KMM}\label{H}
		Let $(M^n,g),n\ge3$ be a complete noncompact Riemannian manifold with nonnegative Ricci curvature and Euclidean volume growth. For $1<p<n$, we have
		$$
		\theta^{\frac pn}\frac{(n-p)^p}{p^p}\int_M\frac{u^p}{r^p}\le\int_M|Du|^p,\quad\forall u\in C_c^\infty(M).
		$$		
	\end{thm}
	
	\begin{proof}
		Let $u^\star$ be the Euclidean rearrangement function of $u$. Through the classical Hardy inequality, we have
		$$
		\frac{(n-p)^p}{p^p}\int_{\mathbb{R}^n}\frac{(u^\star)^q}{|x|^q}\le\int_{\mathbb{R}^n}|Du^\star|^p.
		$$
		Due to the decreasing of $\frac{1}{t^p}$, combining Theorem \ref{PS} and Lemma \ref{sym1}, we have
		$$
		\frac{(n-p)^p}{p^p}\int_M\frac{u^p}{r^p}\le\frac{(n-p)^p}{p^p}\int_{\mathbb{R}^n}\frac{(u^\star)^p}{|x|^p}\le\int_{\mathbb{R}^n}|Du^\star|^q\le\theta^{-\frac pn}\int_M|Du|^p.
		$$
	\end{proof}
	
	Now we give the Heisenberg-Pauli-Weyl inequality. The classical Heisenberg-Pauli-Weyl inequality in Euclidean spaces states that
	$$
	(\int_{\mathbb{R}^n}|x|^2u^2)(\int_{\mathbb{R}^n}|Du|^2)\ge\frac{n^2}{4}(\int_{\mathbb{R}^n}u^2)^2,\quad\forall u\in C_c^\infty(\mathbb{R}^n).
	$$
	By the symmetrization method, one can obtain:
	
	\begin{thm}
		Let $(M^n,g),n\ge3$ be a complete noncompact Riemannian manifold with nonnegative Ricci curvature and Euclidean volume growth. Then
		$$
		(\int_M r^2u^2)
		(\int_M|Du|^2)\ge\theta^\frac 2n\frac{n^2}{4}(\int_M u^2)^2,\quad\forall u\in C_c^\infty(\mathbb{R}^n).
		$$
	\end{thm}

	\begin{proof}
		Let $u^\star$ be the Euclidean rearrangement function of $u$. Through the Heisenberg-Pauli-Weyl inequality in Euclidean space, one can get
		$$
		(\int_{\mathbb{R}^n}|x|^2(u^\star)^2)(\int_{\mathbb{R}^n}|Du^\star|^2)\ge\frac{n^2}{4}(\int_{\mathbb{R}^n}(u^\star)^2)^2.
		$$
		Note that ${t^2}$ is a increasing function, combining Theorem \ref{PS} and Lemma \ref{sym2}, it is easy to show that
		$$
		\begin{aligned}
		\theta^{-\frac2n}(\int_Mr^2(u)^2)(\int_M|Du|^2)&\ge(\int_{\mathbb{R}^n}|x|^2(u^\star)^2)(\int_{\mathbb{R}^n}|Du^\star|^2)\\
		&\ge\frac{n^2}{4}(\int_{\mathbb{R}^n}(u^\star)^2)^2=\frac{n^2}{4}(\int_M u^2)^2.
		\end{aligned}
		$$
	\end{proof}

	\section{Hardy-Sobolev inequality and CKN inequality}
	
	Similar to the previous section, we can obtain the following Hardy-Sobolev inequality.
	
	\begin{thm}\label{HS}
		Let $(M^n,g),n\ge3$ be a complete noncompact Riemannian manifold with nonnegative Ricci curvature and Euclidean volume growth. For $1<p<n$, we have
		$$
		\theta^{\frac{p}{n}}C_{HS}(s,p)(\int_{M}\frac{|u|^{\frac{n-s}{n-p}p}}{r^s})^{\frac{n-p}{n-s}}
		\le\int_{M}|Du|^p,\quad\forall u\in C_c^\infty(M),
		$$
		where $0\le s\le p$ and $C_{HS}(s,p)>0$ is the sharp constant of the corresponding Hardy-Sobolev inequality in Euclidean spaces.
	\end{thm}
	
	\begin{proof}
		The proof is indentity to Theorem \ref{H}, so we omit it here.
	\end{proof}
	
	Through the Hardy-Sobolev inequality above, we have the following CKN inequality:
	
	\begin{thm}
		Let $(M^n,g),n\ge3$ be a complete noncompact Riemannian manifold with nonnegative Ricci curvature and Euclidean volume growth. Then we have
		$$
		C(a,b,\theta)(\int_{M} r^{-bp}|u|^p)^\frac2p\le \int_{M} r^{-2a}|Du|^2,
		$$
		where $0\le a<\frac{n-2}{2}(1-\sqrt{1-\theta^\frac2n}), a\le b\le a+1$, $p=\frac{2n}{n-2+2(b-a)}$, and $C(a,b,\theta)>0$ is a constant depending on $a,b$ and $\theta$.		
	\end{thm}

	\begin{proof}
		The proof is the same as \cite[Corollary 15.1.1]{GM2}. For $u\in C_c^\infty(M)$, let $w(x)=\frac{u(x)}{r^a}$. Through the Stoke's formula and 
		$r\Delta r\le (n-1)$ on $M$, we can obtain
		$$
		\begin{aligned}
		\int_M\frac{|D(wr^a)|^2}{r^{2a}}&=\int_M\frac{1}{r^{2a}}
		(a^2r^{2a-2}w^2+r^{2a}|Dw|^2+2awr^{2a-1}\langle Dr,Dw\rangle)\\
		&=\int_M|Dw|^2+a^2\int_M\frac{w^2}{r^2}+a\int_M\frac{\langle D(w^2),Dr\rangle}{r}\\
		&=\int_M|Dw|^2+\int_Ma(a+1-r\Delta r)\frac{w^2}{r^2}\\
		&\ge\int_M|Dw|^2-a(n-2-a)\int_M\frac{w^2}{r^2}.
		\end{aligned}
		$$		
		Let $\gamma=a(n-2-a)$, by Theorem \ref{H}, one can get
		$$
		\int_M r^{-2a}|Du|^2\ge(1-\gamma\frac{4}{(n-2)^2}\theta^{-\frac{2}{n}})\int_M|Dw|^2>0.
		$$
		Through the Hardy-Sobolev inequality in Theorem \ref{HS}, there exist a constant $C(a,b,\theta)=(\theta^\frac2n-\gamma\frac{4}{(n-2)^2})C_{HS}(\frac{2n(b-a)}{n-2+2(b-a)},2)>0$ such that
		$$
		C(a,b,\theta)(\int_{M} r^{-bp}|u|^p)^\frac2p\le \int_{M} r^{-2a}|Du|^2.
		$$
		
	\end{proof}

\section*{Acknowledgements}

	We would like to thank Prof. A. Krist\'aly for his useful communication. 
	
	\newpage
	\bibliographystyle{plain}
	\bibliography{DLLhardy}

	\noindent\mbox{Yuxin Dong and Lingen Lu} \\
	\mbox{School of Mathematical Sciences}\\
	\mbox{220 Handan Road, Yangpu District}\\
	\mbox{Fudan University}\\
	\mbox{Shanghai, 20043}\\
	\mbox{P.R. China}\\
	\mbox{\textcolor{blue}{yxdong@fudan.edu.cn}}\\
	\mbox{\textcolor{blue}{lulingen@fudan.edu.cn}}
	
	\noindent\mbox{Hezi Lin}\\
	\mbox{School of Mathematics and Statistics \&  FJKLMAA}\\
	\mbox{Fujian Normal University}\\
	\mbox{Fuzhou,  350108}\\
	\mbox{P.R. China}\\
	\mbox{\textcolor{blue}{ lhz1@fjnu.edu.cn}}

\end{document}